\documentclass[12pt]{article}
\usepackage{a4wide}
\usepackage{amsmath,amssymb,amsthm}
\usepackage[mathscr]{eucal}
\usepackage{tikz}
\usetikzlibrary{positioning,automata}
\usetikzlibrary{arrows,calc}
\usetikzlibrary{decorations.markings,plotmarks}
\usepackage{graphics}
\usepackage{color}
\usepackage{leftidx}
\usepackage{hyperref}

\definecolor{red}{rgb}{1,0.00,0.00}

\newtheorem{theorem}{Theorem}
\newtheorem{proposition}[theorem]{Proposition}

\theoremstyle{definition}

\author{Andriy Oliynyk and Veronika Prokhorchuk}

\title{\textbf{On exponentiation, $p$-automata and HNN extensions of free abelian groups}}

\date{}

\begin{document}

\maketitle

\begin{abstract}
    For every prime $p$ it is shown that a wide class of HNN extensions of free abelian groups admit faithful representation by finite $p$-automata.
\end{abstract}

\section{Introduction}
\label{section_Introduction}

Natural action of the wreath product of permutation groups $(G,\mathsf{X})$ and $(H,\mathsf{Y})$ on the Cartesian product $\mathsf{X} \times \mathsf{Y}$ is widely used. The other action of this wreath product on the set $\mathsf{Y}^{\mathsf{X}}$ of all functions from $\mathsf{X}$ to $\mathsf{Y}$ is called exponentiation.
This action was defined in~\cite{MR0107666} as a formalization of group actions used in~\cite{MR0002510} to enumerate types of Boolean functions (cf.~\cite{MR0409266,MR2604956}).  
Its basic properties as a permutation group were obtained in~\cite{MR0554500}. Exponentiation was applied to construct and study new strongly regular graphs (\cite{MR0603646,MR0617207}), to study automorphism group of the {$n$}-dimensional cube (\cite{MR1801422,MR1897935}), to construct new finite Gelfand pairs (\cite{MR2263713}) and to construct new finitely generated profinite groups (\cite{MR3391479}). 
We observe that exponentiation can be applied to construct groups defined by finite automata.

Finite automata and groups defined by automata form a valuble direction in modern mathematics (see e.g.~\cite{MR1841755,MR2162164}). Given a finite (permutational) automaton over a finite alphabet $\mathsf{X}$ a naturally defined permutation group, the group of this automaton, acting by automorphisms on free monoid $\mathsf{X}^*$ as on a rooted tree is defined. The other way to define a group by an automaton is to generate it by a subset of the generating set of the group of this automaton. Given a group defined by an automaton over some alphabet it is naturally to decrease size of the alphabet such that this group can be defined by an automaton over minimized alphabet. Moreover, additional restrictions on permutations defined at states of automata can be applied (\cite{MR2796035}).

In the present paper we consider this problem for HNN extensions of free abelian groups to extend results of~\cite{MR4296450}. We use automata constructed in~\cite{MR2216703} such that their groups are required HNN extensions. As the main result for any prime $p$ we give sufficient conditions on an ascending HNN extension of a free abelian group of finite rank under which this HNN extension can be defined by a finite automaton over an alphabet of cardinality $p$ and all permutations on the alphabet at the states are powers of a certain cycle of length $p$.

The paper organized as follows. In Section~\ref{section_wreath_products} we recall required definitions regarding wreath products and exponentiation. We observe in Theorem~\ref{thm_exponentiation_acts_on_tree} a sufficient condition to represent an exponentiation as a permutation group in terms of a permutational wreath product acting on a finite rooted tree. Here we also give a constructive example of such a representation. In Section~\ref{section_automata_and_groups_defined_by_them} we briefly recall required notions on finite automata and groups defined by them. In Theorem~\ref{thm_minimize_group_at_states} we give a sufficient condition for a group defined by an automaton to decrease the size of the alphabet such that it can be defined by an automaton over minimized alphabet. In Section~\ref{section_hnn_actions_defined_by_automata} we use these statements to prove the main result of the paper.

All used notions and properties about trees, automata and groups are standard and can be found e.g. in~\cite{MR1841755,MR2162164}.

\section{Wreath products and their actions}
\label{section_wreath_products}
\subsection{Wreath products}

Let $(G,\mathsf{X})$ and $(H,\mathsf{Y})$ be permutation groups. The wreath product of  $(G,\mathsf{X})$ with $H$ is defined as the semidirect product 
\[
G \ltimes H^\mathsf{X}
\]
where the action of $G$ on the set of functions $H^\mathsf{X}$ is induced by its action on $\mathsf{X}$.
It is denoted by $G \wr H$ and  consists of the pairs
$[g, h(x)]$, where $g \in G$, $h(x): \mathsf{X} \to H$.
Such a pair acts on the Cartesian product $\mathsf{X} \times \mathsf{Y}$ by the rule
\[
(x,y)^{[g,h(x)]}=(x^g,y^{h(x)}), \quad x\in \mathsf{X}, y \in \mathsf{Y}.
\]
The permutation group $(G \wr H, \mathsf{X} \times \mathsf{Y})$ is called the permutational wreath product of $(G,\mathsf{X})$ and $(H,\mathsf{Y})$.

Being an associative operation on permutation groups their wreath product can be defined for arbitrary finite sequence 
\[
(G_1,\mathsf{X}_1), (G_2,\mathsf{X}_2), \ldots , (G_n,\mathsf{X}_n)
\]
of permutation groups. In this case it is denoted by $\wr_{i=1}^n G_i$ and consists of tuples
\[
[g_1,g_2(x_1),\ldots , g_n(x_1,\ldots,x_{n-1})],
\]
where 
\[
g_1 \in G_1, \quad g_2(x_1):\mathsf{X}_1 \to G_2,  \ldots,  g_n(x_1,\ldots,x_{n-1}):\mathsf{X}_1 \times \ldots \times \mathsf{X}_{n-1} \to G_n.
\]
It acts on the sets $\varnothing$, $\mathsf{X}_1$, $\mathsf{X}_1 \times \mathsf{X}_2$, $\ldots ,\mathsf{X}_1 \times \ldots \times \mathsf{X}_{n}$ preserving the natural structure of a rooted tree on their union. Hence, the permutational wreath product $\wr_{i=1}^n G_i$ can be viewed as an automorphism group of a homogeneous rooted tree acting on the set of its leaves.

\subsection{Exponentiation}

For permutation groups $(G,\mathsf{X})$ and $(H,\mathsf{Y})$ the action of the wreath product of $(G,\mathsf{X})$ with $H$ on the set $\mathsf{Y}^{\mathsf{X}}$ can be defined as follows.

The exponentiation of $(H,\mathsf{Y})$ by $(G,\mathsf{X})$ is the permutation group
\[
H \uparrow G = (G \wr H, \mathsf{Y}^{\mathsf{X}})
\]
such that every $[g,h(x)] \in G \wr H$ acts on $f(t) \in \mathsf{Y}^{\mathsf{X}}$ by the rule
\[
f(t)^{[g,h(x)]} = (f(t^g))^{h(t)}.
\]

Since for finite $\mathsf{X}$ and $\mathsf{Y}$ the degrees of permutation groups $(G \wr H, \mathsf{X} \times \mathsf{Y})$ and $(G \wr H, \mathsf{Y}^{\mathsf{X}})$ are not equal these groups are not isomorphic as permutation groups. Moreover, exponentiation is not an associative operation on permutation groups. However, it is natural to ask about existence of an isomorphism between the permutation group $H \uparrow G$ and an automorphism group of a homogeneous rooted tree acting on the set of its leaves. 

In general such an isomorphism does not exist. For instance, consider the exponentiation $\mathbb{Z}_2 \uparrow \mathbb{Z}_3$ of the regular cyclic group of order 2 by the regular cyclic group of order 3. It is a permutation group of degree $8$ and contains elements of order 3. From the other hand the automorphism group of the 2-regular rooted tree of depth $3$ has exactly 8 leaves and no elements of order 3.

We give a sufficient condition under which a required isomorphism exists.

\begin{theorem}
\label{thm_exponentiation_acts_on_tree}
    Let $p$ be a prime, $G$ and $H$ be finite $p$-groups faithfully acting on sets $\mathsf{X}$ and $\mathsf{Y}$ of cardinalities $p^n$ and $p^m$ correspondingly, $n,m\ge 0$. Then the exponentiation $H \uparrow G$ is isomorphic as a permutation group to the wreath product of $m \cdot p^n$ copies of the regular cyclic group or order $p$ acting by automorphisms on the set of leaves of the $p$-regular rooted tree of depth $m\cdot p^n$.
\end{theorem}

\begin{proof}
    Under conditions of the Theorem the exponentiation $H \uparrow G$ is a permutation group of order $p^{n+m\cdot p^{n}}$ and degree $p^{m\cdot p^{n}}$. Then it can be considered as a $p$-subgroup of the symmetric group of degree $p^{m\cdot p^{n}}$. Hence, by Sylow's Theorem it is contained in a Sylow $p$-subgroup of this symmetric group. Sylow $p$-subgroup of the symmetric group of degree $p^{m\cdot p^{n}}$ is isomorphic to the wreath product of $m \cdot p^n$ copies of the regular cyclic group or order $p$ (\cite{MR0014087}). The statement immediately follows.
\end{proof}

\subsection{Example}

Let $p=3$. Consider the additive group $\mathbb{Z}_3=\{0,1,2\}$ of residues modulo 3 regularly acting on itself. Theorem~\ref{thm_exponentiation_acts_on_tree} implies that the exponentiation $\mathbb{Z}_3 \uparrow \mathbb{Z}_3$ as a permutation group is isomorphic to a subgroup of the permutation group
\[
(\mathbb{Z}_3 \wr \mathbb{Z}_3 \wr \mathbb{Z}_3, \mathbb{Z}_3^3).
\]
An example of an isomorphism can be constructed as follows.  

Let
\[
a = [1, (0,0,0)], \quad b = (0, (1,0,0))]
\]
be elements of the wreath product $\mathbb{Z}_3 \wr \mathbb{Z}_3$. They form a generating set of this group. Denote by $\sigma_a$ and $\sigma_b $ permutations on $\mathbb{Z}_3^3$ defined by $a$ and $b$ correspondingly.

Consider elements 
\[
c = [c_1,c_2(x_1),c_3(x_1,x_2)],\quad  d = [d_1,d_2(x_1),d_3(x_1,x_2)] \in \mathbb{Z}_3 \wr \mathbb{Z}_3 \wr \mathbb{Z}_3
\]
such that 
\[
c_1=0, \quad
c_2(x_1)=
\begin{cases}
    2, & \text{ if } x_1=0 \\
    1, & \text{ if } x_1=1 \\
    0, & \text{ if } x_1=2
\end{cases}, \quad
c_3(x_1,x_2)=
\begin{cases}
    2, & \text{ if } x_1=2, x_2=0 \\
    1, & \text{ if } x_1=2, x_2 =1 \\
    0, & \text{ otherwise }
\end{cases},
\]
\[
d_1=1, \quad
d_2(x_1)=0, x_1 \in \mathbb{Z}_3, \quad
d_3(x_1,x_2)=
\begin{cases}
    1, & \text{ if } x_1=0, x_2\ne 2 \\    
    0, & \text{ otherwise }
\end{cases}.
\]
Denote by $\sigma_c$ and $\sigma_d$ permutations on $\mathbb{Z}_3^3$ defined by $c$ and $d$ correspondingly.

Define a bijection $\pi : \mathbb{Z}_3^3 \to \mathbb{Z}_3^3$ as follows:
\begin{align*}
(0,0,0) & \mapsto (0,0,1), & (0,0,1) & \mapsto (1,1,2), & (0,0,2) & \mapsto (2,2,0), \\
(0,1,0) & \mapsto (0,1,0), & (0,1,1) & \mapsto (1,2,1), & (0,1,2) & \mapsto (2,0,2), \\
(0,2,0) & \mapsto (1,0,0), & (0,2,1) & \mapsto (2,1,1), & (0,2,2) & \mapsto (0,2,2), \\
(1,0,0) & \mapsto (0,2,0), & (1,0,1) & \mapsto (1,0,1), & (1,0,2) & \mapsto (2,1,2), \\
(1,1,0) & \mapsto (0,0,2), & (1,1,1) & \mapsto (1,1,0), & (1,1,2) & \mapsto (2,2,1), \\
(1,2,0) & \mapsto (2,0,0), & (1,2,1) & \mapsto (0,1,1), & (1,2,2) & \mapsto (1,2,2), \\
(2,0,0) & \mapsto (1,2,0), & (2,0,1) & \mapsto (2,0,1), & (2,0,2) & \mapsto (0,1,2), \\
(2,1,0) & \mapsto (1,0,2), & (2,1,1) & \mapsto (2,1,0), & (2,1,2) & \mapsto (0,2,1), \\
(2,2,0) & \mapsto (0,0,0), & (2,2,1) & \mapsto (1,1,1), & (2,2,2) & \mapsto (2,2,2).
\end{align*}

Then for arbitrary $\alpha \in \mathbb{Z}_3^3$ the following equalities hold:
\[
\sigma_a(\alpha) = \pi (\sigma_c( \pi^{-1} (\alpha))), \quad 
\sigma_b(\alpha) = \pi (\sigma_d( \pi^{-1} (\alpha))).
\]
The required isomorphism now directly follows.

\section{Automata and groups defined by automata}
\label{section_automata_and_groups_defined_by_them}

\subsection{Words and automata}

Let $\mathsf{X}$ be a finite set, $|\mathsf{X}|>1$. The set $\mathsf{X}^*=\cup_{i=0}^{\infty}\mathsf{X}^i$ of all finite words over $\mathsf{X}$, including the empty word $\Lambda$, form a free monoid with basis $\mathsf{X}$ under concatenation. The length of a word $w \in \mathsf{X}^*$ is denoted by $|w|$, i.e. $w \in \mathsf{X}^{|w|}$. The right Cayley graph of $\mathsf{X}^*$ with respect to basis $\mathsf{X}$ defines on $\mathsf{X}^*$ as on the vertex set the structure of a regular tooted tree. Two words $u,v \in \mathsf{X}^*$ are joined by an edge if and only if $u = vx$ or $v =ux$ for some $x \in \mathsf{X}$.  For every $n\ge 0$ the set $\mathsf{X}^n$ form the $n$th level of this tree and  the union $\cup_{i=0}^{n} \mathsf{X}^n$ is the vertex set of a regular rooted subtree of depth $n$.

A finite automaton over alphabet $\mathsf{X}$ is a triple $\mathcal{A}=(Q, \lambda, \mu)$ such that $Q$ is a finite non-empty set, the set of states, $\lambda : Q \times \mathsf{X} \to Q$ is the transition function, $\mu : Q \times \mathsf{X} \to Q$ is the output function. Automaton $\mathcal{A}$ is called permutational if for every $q \in Q$ the restriction $\mu_q : \mathsf{X} \to \mathsf{X}$ of the output function at state $q$ is a permutation on $\mathsf{X}$. We will consider finite permutational automata only.

In case $|\mathsf{X}|=p$ for some prime $p$ and all permutations $\mu_q$, $q\in Q$, are powers of a fixed cycle of length $p$ on $\mathsf{X}$ automaton $\mathcal{A}$ is called $p$-automaton.

\subsection{Groups defined by automata}

Let $\mathcal{A}=(Q, \lambda, \mu)$ be a finite permutational automaton over $\mathsf{X}$. 
The set of permutations $\{\mu_q, q\in Q\}$ generate a subgroup in the symmetric group on $\mathsf{X}$.  We call this group the  permutation group defined at states of $\mathcal{A}$. For a $p$-automaton the  permutation group defined at its states is the regular cyclic group of order $p$.

For every $q \in Q$ the permutation $\mu_q$ can be recursively extended to the permutation on the set $\mathsf{X}^*$ as follows:
\[
\mu_q(\Lambda)=\Lambda, \quad \mu_q(xw)=\mu_q(x)\mu_{\lambda(q,x)}(w), \quad x \in \mathsf{X}, w \in \mathsf{X}^*.
\]
Obtained in this way permutation $\mu_q$ is  length preserving, i.e. $|\mu_q(w)|=|w|$, $w \in \mathsf{X}^*$,  and prefix preserving, i.e. if for $w,w_1 \in \mathsf{X}^*$ and some $x \in \mathsf{X}$ the equality $w=w_1x$ holds then some $x_1 \in \mathsf{X}$ the equality $\mu_q(w)=\mu_q(w_1)x_1$  holds.  Hence, $\mu_q$  preserves the structure of a rooted tree on $\mathsf{X}^*$. It is called an automaton permutation defined by $\mathcal{A}$ at state $q$.  

The permutation group, generated by the set $\{\mu_q, q \in Q\}$ is called the group of the automaton $\mathcal{A}$ and denoted by $G(\mathcal{A})$. The restriction of its action on $\mathsf{X}$ is the  permutation group defined at states of $\mathcal{A}$. 

More generally, a group $G$ is called a group generated by a finite automaton over $\mathsf{X}$ if there exists a finite permutational automaton $\mathcal{A}$ over $\mathsf{X}$ such that $G$ is isomorphic to the group generated by a subset of automaton permutations defined at states of $\mathcal{A}$. In this case $G$ is isomorphic to a subgroup of automaton $\mathcal{A}$ generated by a subset of its generating set. 

In terms of \cite{MR2796035} a group generated by a finite automaton over $\mathsf{X}$ is a finitely generated subgroup of the finite state wreath power of the permutation group at states of an automaton over $\mathsf{X}$. In some cases the order and the degree of such a permutation group can be minimized.

\begin{theorem}
    \label{thm_minimize_group_at_states}
    Let $(G,\mathsf{X})$ and $(H,\mathsf{Y})$ be finite permutation groups such that $(G,\mathsf{X})$ is isomorphic as a permutation group to a subgroup of the wreath product of finitely many copies of $(H,\mathsf{Y})$. Then every group generated by a finite automaton such that the permutation group defined at its states is $(G,\mathsf{X})$ can be generated by a finite automaton, the permutation group defined at its states is $(H,\mathsf{Y})$.
\end{theorem}

\begin{proof}
Assume that $(G,\mathsf{X})$ is isomorphic as a permutation group to a subgroup of the wreath product of $r$  copies of $(H,\mathsf{Y})$ for some $r \ge 1$. Denote by $\varphi$ an isomorphic embedding of $G$ into $\wr_{i=1}^{r} H^{(i)} $, $H^{(i)} \simeq H$, $1\le i \le r$,  and by $\psi$ an injection from $\mathsf{X}$ to $\mathsf{Y}^r$ such that 
\[
\psi(x^g)=(\psi(x))^{\varphi(g)}, \quad x \in \mathsf{X}, g \in G.
\]
Denote by $\mathsf{Y}_1$ the image of $\mathsf{X}$ under $\psi$. 
Then $(G,\mathsf{X})$ is isomorphic as a permutation group to $(\varphi(G), \mathsf{Y}_1)$. Recall that the wreath product $\wr_{i=1}^{r} H^{(i)}$ acts on the union 
$\cup_{i=0}^{r-1}\mathsf{Y}^{i}$. This action is length preserving and prefix preserving. Hence, $\varphi(G)$ acts on the set $\mathsf{Y}_2$ that consists of all prefixes of all words from $\mathsf{Y}_1$.

Let $\mathcal{A}=(Q, \lambda, \mu)$ be a finite permutational automaton over $\mathsf{X}$ such that the permutation group defined at its states is $(G,\mathsf{X})$. It is sufficient to show that the group $G(\mathcal{A})$ of this automaton can be  generated by a finite automaton over $\mathsf{Y}$ such that the permutation group defined at its states is $(H,\mathsf{Y})$. We define a corresponding automaton $\mathcal{B}= (Q_1,\lambda_1, \mu_1)$.

The set of sates $Q_1$ of $\mathcal{B}$ is the set of all possible pairs of the form
$(q,w)$, where $q$ is a state of $\mathcal{A}$ and $w$ is a word from $\mathsf{Y}$ of length not greater than $r-1$. In other words, it is defined as the Cartesian product
\[
Q_1 = Q \times \left( \cup_{i=0}^{r-1}\mathsf{Y}^{i} \right).
\]

The transition function $\lambda_1$ is defined by the equality
\[
\lambda_1((q,w),y)=
\begin{cases}
    (q,wy), & \text{ if } |w|<r-1 \text{ and } wy \in \mathsf{Y}_2 \\
    (\lambda(q,\psi^{-1}(wy)),\Lambda),&  \text{ if } |w|=r-1 \text{ and } wy \in \mathsf{Y}_1\\
    (q,w),&  \text{ otherwise }
\end{cases}.
\]
Since $\psi$ is an injection the definition is correct.

The output function $\mu_1$ is defined by the equality
\[
\mu_1((q,w),y)=
\begin{cases}
    y_1, & \text{ if }  wy \in \mathsf{Y}_2 \text{ and } (wy)^{\varphi(\mu_q)}=
    (w)^{\varphi(\mu_q)}y_1\\
    y,&  \text{ otherwise }
\end{cases}.
\]
Since $\varphi(\mu_q)$, $q \in Q$, is length preserving and prefix preserving on $\mathsf{Y}_2$ the definition is correct.

It is required to find a subset $S \subset Q_1$ such that the group $G(\mathcal{A})$ is isomorphic to the group $G_S$ generated by the set $\{{\mu_1}_s, s \in S \}$. We will show that one can take the subset $S=\{(q,\Lambda), q\in Q \}$. It is enough to prove that the mapping $\mu_q \mapsto {\mu_1}_{(q,\Lambda)}$, $q \in Q$, defines a required isomorphism between $G(\mathcal{A})$ and $G_S$.

Consider the monoid monomorphism $\Phi: \mathsf{X}^* \to \mathsf{Y}^*$ that extends injection $\varphi$. Then the image $\Phi(\mathsf{X}^*)$ is a free monoid with basis 
$\mathsf{Y}_1$, i.e. $\Phi(\mathsf{X}^*)=\mathsf{Y}_1^*$. 

For arbitrary $q \in Q$, $x \in \mathsf{X}$ and $w \in \mathsf{X}^*$ we have the equalities
\[
\Phi(xw)^{{\mu_1}_{(q,\Lambda)}}=(\varphi(x)\Phi(w))^{{\mu_1}_{(q,\Lambda)}}=
(\varphi(x))^{{\mu_1}_{(q,\Lambda)}}\Phi(w)^{{\mu_1}_{(\lambda(q,x),\Lambda)}}=
\varphi(x^{{\mu}_{q}})\Phi(w)^{{\mu_1}_{(\lambda(q,x),\Lambda)}}.
\]
Since $\varphi(x) \in \mathsf{Y}_1$  the last equality follows from the definition of the output function $\mu_1$. 
Therefore, permutations groups  $(G(\mathcal{A}), \mathsf{X}^*)$ and $(G_S,\mathsf{Y}_1^*)$ are isomorphic as permutation groups.

Consider arbitrary $w\in \mathsf{Y}^*$. Then there exist unique 
$w_1\in \mathsf{Y}_1^*$, $w_2\in \mathsf{Y}_2 \setminus \mathsf{Y}_1$, $w_3\in \mathsf{Y}^*$ such that
$w=w_1w_2w_3$ and the word $w_2$ is the longest prefix of $w_2w_3$ from $\mathsf{Y}_2$.
Let $q \in Q$. Then
\[
w^{{\mu_1}_{(q,\Lambda)}}=w_1^{{\mu_1}_{(q,\Lambda)}}w_4w_3 
\]
for some $w_4 \in \mathsf{Y}_2 \setminus \mathsf{Y}_1$ such that for arbitrary $w_1w_2u \in \mathsf{Y}_1^*$, $u \in \mathsf{Y}^*$, the word $w_1^{{\mu_1}_{(q,\Lambda)}}w_4$ is a prefix of $(w_1w_2u)^{{\mu_1}_{(q,\Lambda)}}$.
Hence, the action of the automaton permutation ${\mu_1}_{(q,\Lambda)}$ on $\mathsf{Y}^*$
is completely defined by its action on $\mathsf{Y}_1^*$.
The proof is complete.
\end{proof}

\section{HNN extensions and $p$-automata}
\label{section_hnn_actions_defined_by_automata}

\subsection{HNN extensions of free abelian groups}

Let  $A_r = \langle a_1,\ldots, a_r \mid 
        a_ia_j=a_ja_i, 1\le i<j\le r
        \rangle $ be a free abelian group of rank $r\ge 1$.
For a non-degenerate integer matrix $M=(m_{ij})_{i,j=1}^{r}$
consider the group         
\[
\mathbb{G}_{M}=\langle A_r, t \mid  
a_i^{t}=a_1^{m_{i1}} \ldots a_r^{m_{ir}}, 1 \le i \le r  \rangle.
\]
Then $\mathbb{G}_{M}$ is an ascending HNN extension of $A_r$.

\begin{proposition}[\cite{MR2216703}]
\label{bartholdi_sunik_theorem}
Let the order of $M$ is infinite and for positive integer $n\ge 2$ the determinant of $M$ is relatively prime to $n$. Then there exist a finite permutational automaton $\mathcal{A}_M$ over $\mathsf{X}=\{0,\ldots , n-1 \}^r$ such that the group of $\mathcal{A}_M$ is isomorphic to $\mathbb{G}_{M}$.
\end{proposition}

Let us recall the construction of automaton $\mathcal{A}_M= (Q, \lambda, \mu)$ from~\cite{MR2216703}. Denote by $n(M)$ the max norm of the matrix $M$, i.e.
\[
n(M)=\max_{i}\sum_{j}|m_{ij}|.
\]
Then the set of states $Q$ is defined as 
\[
Q=\{(v_1,\ldots, v_r)^{\top} \in \mathbb{Z}^r : -n(M)\le v_i \le n(M)-1, 1\le i \le r\}
\]
and for $v=(v_1,\ldots, v_r)^{\top} \in Q$, $x=(x_1,\ldots, x_r)^{\top} \in \mathsf{X}$ we have
\[
\lambda(v,x)=\mathrm{Div}_n(v+Mx), \quad 
\mu(v,x)=\mathrm{Mod}_n(v+Mx),
\]
where $\mathrm{Div}_n$ and $\mathrm{Mod}_n$ denote operations of taking coordinate-wise quotients and reminders from division by $n$.

\subsection{${p}$-automata defining HNN extensions}

The main result of this section is the following 

\begin{theorem}
    Let $p$ be a prime, $r = p^k$ for some $k \ge 1$ and integer $r\times r$ matrix $M$ 
    has the form     
    \[M=pN+C
    \]
    for some $r \times r$ integer matrix $N$ and permutation matrix $C$ that correspond to a permutation of order $p^m$ for some $m\ge 0$. Then 
    \begin{enumerate}
        \item 
        the group $\mathbb{G}_{M}$ is generated by a finite $p$-automaton;
        \item 
        the group $\mathbb{G}_{M}$ is residually $p$-finite.
    \end{enumerate}
\end{theorem}

\begin{proof}
Since $|\det M|=1$ this determinant is relatively prime to $p$ and matrix $M$ satisfy conditions of Proposition~\ref{bartholdi_sunik_theorem}. Then the group of the automaton $\mathcal{A}_M$ is isomorphic to $\mathbb{G}_{M}$. Let us describe the permutation group $(G,\mathsf{X})$ defined at states of this automaton.

Denote by $\mathsf{X}_1$ the set $\{1,\ldots,r\}$ and by $\mathsf{X}_2$ the set $\{0,\ldots,p-1\}$. Then the alphabet $\mathsf{X}$ can be identified with the set $\mathsf{X}_2^{\mathsf{X}_1}$ of all functions from $\mathsf{X}_1$ to $\mathsf{X}_2$.

Denote by $\sigma$ the permutation on $\mathsf{X}_1$ such that the matrix $C$ corresponds to $\sigma$. Since $\sigma$ has order $p^m$ the lengths of its independent cycles are powers of $p$ not greater than $p^m$ and at least one of them is $p^m$. 
Denote by $G_1$ the cyclic group generated by $\sigma $. Then $|G_1|=p^m$.

Let $G_2$ be a cyclic group of order $p$. It acts on $\mathsf{X}_2$ by additions modulo $p$.

Consider arbitrary state $v=(v_1,\ldots, v_r)^{\top}$ of $\mathcal{A}_M$. Then the vector
$\mathrm{Mod}_p(v)$ can be regarded as a function from $\mathsf{X}_1$ to $G_2$. For arbitrary vector $x=(x_1,\ldots, x_r)^{\top} \in \mathsf{X}$ we have
\[
\mu_v(x)=\mathrm{Mod}_p(v+Mx)=\mathrm{Mod}_p(v+(pN+C)x)=\mathrm{Mod}_p(v+Cx)
\]
and coordinate-wise 
\[
\mu_v(x)=((x_{\sigma(i)}+v_i) \mod{p} , 1\le i \le r).
\]
It means that the permutation $\mu_v$ acts on $\mathsf{X}_2^{\mathsf{X}_1}$ by the rule,  
that defines a permutation from the exponentiation $G_2 \uparrow G_1$, i.e. the permutation $\mu_v$ is defined by  the element $[\sigma, \mathrm{Mod}_p(v)]$ of the wreath product $G_1 \wr G_2$, the wreath product of cyclic groups of orders $p^{m}$ and $p$.
Since $n(M)\ge 1$ all vectors 
\[
u_1=(-1,0,\ldots, 0), \ldots, u_r=(0,\ldots , 0,-1) 
\]
belong to the set $Q$ of states of $\mathcal{A}_M$. Hence, the set 
\[
\{
\mu_{u_1}, \ldots \mu_{u_r}
\}
\]
generates the wreath product $G_1 \wr G_2$. Therefore, the permutation group  $(G,\mathsf{X})$ is the exponentiation $G_2 \uparrow G_1$.

Theorem~\ref{thm_exponentiation_acts_on_tree} now implies that  $(G,\mathsf{X})$ is isomorphic as a permutation group to the wreath product of  copies of the regular
cyclic group or order $p^m$ acting by automorphisms on the set of leaves of the $p$-regular rooted tree of depth $p^m$. Then Theorem~\ref{thm_minimize_group_at_states} implies that the group $\mathbb{G}_{M}$ can be generated by a finite automaton with regular cyclic group as the permutation group defined at its states, i.e. by a $p$-automaton. This completes the proof of the first statement of the theorem. 

Since all groups defined by finite $p$-automata are residually $p$-finite (see e.g.~\cite{MR4296450}) the second statement now immediately follows. 
\end{proof}

\section*{Acknowledgements}
The research presented in the paper was done during the fellowship of the second author at the Institute of Mathematics of the Polish Academy of Sciences supported by 
the European Research Council (ERC) under the European Union’sHorizon 2020 research and innovation programme (Grant Agreement~No.~677120-INDEX)
and
Grant Norweski UMO-2022/01/4/ST1/00026. 

\bibliographystyle{amsplain}
\bibliography{references}

\end{document}